\documentclass[10pt,reqno]{amsart}

\usepackage{amsmath,amssymb,amsfonts,amsthm,mathrsfs,enumitem}
\allowdisplaybreaks

\newtheorem{thm}{Theorem}[section]

\newtheorem{cor}[thm]{Corollary}
\newtheorem{lem}[thm]{Lemma}
\newtheorem{prop}[thm]{Proposition}
\newtheorem{con}[thm]{Conjecture}
\theoremstyle{definition}

\theoremstyle{remark}
\newtheorem{rem}[thm]{Remark}

\numberwithin{equation}{section}

\newcommand{\fq}{{\mathbb F}_{q}}
\newcommand{\fqq}{{\mathbb F}_{q^2}}
\newcommand{\ftwo}{{\mathbb F}_{2}}
\newcommand{\Tr}{\operatorname{Tr}}

\newcommand{\abs}[1]{\left|#1\right|}
\newcommand{\set}[1]{\left\{#1\right\}}

\begin{document}

\title[A proof of a permutation-inverse conjecture]{A Proof of a Permutation-Inverse Bent-Function Conjecture}
\author{Kaimin Cheng}
\address{School of Mathematical Sciences, China West Normal University, Nanchong 637002, P. R. China}
\email{ckm20@126.com}
\date{}

\begin{abstract}
Let $q=2^e$ with $e$ even, and let $\fqq$ be the finite field of order $q^2$. Put $d=(q^2+q+1)/3$, and consider the permutation polynomial
\[
\sigma(X)=X+X^d+X^{dq}\in\fqq[X].
\]
For $\alpha\in\fq^*$, define
\[
f_{\alpha}(x)=\Tr_{q^2}\bigl(\alpha(\sigma^{-1}(x))^3\bigr),\qquad x\in\fqq.
\]
We prove that $f_{\alpha}$ is bent if and only if $\alpha$ is not a cube in $\fq$, thereby proving a conjecture of Li, Li, Helleseth, and Qu. The proof computes the Walsh values on $\fq$ directly and treats the complementary parameters by reducing them to a two-variable exponential sum. A binary Hasse congruence, proved by a finite carry analysis together with a projective-frame cancellation for the only large recurrent component, forces the outside Walsh coefficients in the noncubic case to be $\pm q$. As an application, we identify a recent cyclotomic family of Xie, Li, Wang, and Zeng with the same construction in different coordinates and thereby prove their conjecture.
\end{abstract}

\maketitle

\section{Introduction}

Let $q=2^e$ with $e$ a positive integer, let $\fqq$ denote the finite field of order $q^2$, and write $\Tr_q$ and $\Tr_{q^2}$ for the absolute traces from $\fq$ and $\fqq$ to $\ftwo$, respectively.
For a Boolean function $f$ on $\fqq$, its Walsh transform is
\[
W_f(\beta)=\sum_{x\in\fqq}(-1)^{f(x)+\Tr_{q^2}(\beta x)},\qquad \beta\in\fqq.
\]
The function $f$ is called \emph{bent} if $\abs{W_f(\beta)}=q$ for every $\beta\in\fqq$.
Bent functions are extremal objects in the theory of Boolean functions and coding theory; for background we refer to \cite{CM16,Mes16}.
For constructions defined through inverses of permutation polynomials, however, the inverse map usually conceals the algebraic structure that one would like to exploit in the Walsh transform.
This makes explicit spectral results rare even when the underlying bentness question looks elementary.

We study such a family in the even-characteristic setting.
Assume from now on that $e$ is even, put
\[
d=\frac{q^2+q+1}{3},
\]
and let
\[
\sigma(X)=X+X^d+X^{dq}\in\fqq[X].
\]
Since $e$ is even, we have $q=2^e\equiv 1\pmod 3$. In particular,
\begin{equation}\label{eq:gcd-facts}
\gcd(3,q+1)=1,
\qquad
\gcd(3,q-1)=3.
\end{equation}
We shall use \eqref{eq:gcd-facts} repeatedly without further comment.

Ding, Qu, Wang, Yuan, and Yuan \cite{Ding2015} proved in 2015 that $\sigma$ is a permutation polynomial of $\fqq$.
For each $\alpha\in\fq^*$, define
\begin{equation}\label{eq:def-falpha}
 f_{\alpha}(x)=\Tr_{q^2}\bigl(\alpha(\sigma^{-1}(x))^3\bigr),\qquad x\in\fqq.
\end{equation}
In 2023, Li, Li, Helleseth, and Qu \cite{Li2023} proposed the following conjecture.

\begin{con}\label{con:main}
Let $q=2^e$ with $e$ an even positive integer.
For $\alpha\in\fq^*$, let $f_{\alpha}$ be the Boolean function defined by \eqref{eq:def-falpha}.
Then $f_{\alpha}$ is bent over $\fqq$ if and only if $\alpha$ is not a cube in $\fq$.
\end{con}

Our main result is the following bentness criterion. We keep a little more information than the conjecture asks for: the entire part of the Walsh spectrum indexed by the subfield $\fq$ is computed, and in the noncubic case all outside Walsh values are shown to be $\pm q$.

\begin{thm}\label{thm:main-Fq}
Let $\alpha\in\fq^*$, and let $f_{\alpha}$ be the Boolean function defined by \eqref{eq:def-falpha}.
Then the Walsh values of $f_{\alpha}$ at points of $\fq$ are as follows.
\begin{enumerate}[label=\textup{(\alph*)},leftmargin=*]
\item If $\alpha$ is a cube in $\fq$, then
\[
\set{W_{f_{\alpha}}(\beta):\ \beta\in\fq}=\set{2q,-2q}.
\]
\item If $\alpha$ is not a cube in $\fq$, then
\[
\set{W_{f_{\alpha}}(\beta):\ \beta\in\fq}=\set{q}.
\]
\end{enumerate}
\end{thm}

\begin{thm}
\label{thm:main-outside}
Let $\alpha\in\fq^*$ be a noncube, and let $f_{\alpha}$ be defined by \eqref{eq:def-falpha}. Then
\[
\set{W_{f_{\alpha}}(\beta):\ \beta\in\fqq\setminus\fq}=\{q,-q\}.
\]
\end{thm}

Together, Theorems~\ref{thm:main-Fq} and~\ref{thm:main-outside} prove the bentness direction for noncubic $\alpha$. The cubic case is already excluded by Theorem~\ref{thm:main-Fq}\textup{(a)}, which gives Walsh coefficients of absolute value $2q$ on the subfield. Thus the paper does not need, and does not claim, a complete outside spectrum in the cubic case.

\begin{cor}\label{cor:distribution}
If $\alpha\in\fq^*$ is not a cube, then the Walsh transform of $f_{\alpha}$ takes the value $q$ with multiplicity $q(q+1)/2$ and the value $-q$ with multiplicity $q(q-1)/2$.
\end{cor}

\begin{cor}\label{cor:conj}
Conjecture~\ref{con:main} is true.
\end{cor}

The proof splits naturally according to whether the spectral parameter $\beta$ lies in $\fq$ or in $\fqq\setminus\fq$.
The $\beta\in\fq$ regime yields to an elementary finite-field computation.
For $\beta\in\fqq\setminus\fq$, an explicit parametrization of $\mu_{q+1}$ reduces the problem to exponential sums on $\fq$.
After a denominator-free change of variables, the noncubic outside case is controlled by a two-variable Hasse congruence modulo $2q$.
The carry graph of the Hasse term has two small components, which give $1$ and the cubic-character expression $\delta^{(q-1)/3}+\delta^{2(q-1)/3}$, and one large component.
The large component is killed by an Artin--Schreier boundary relation $\eta^q+\eta=1$ through a moving projective frame.

The final section shows that a recent cyclotomic construction of Xie, Li, Wang, and Zeng gives the same Boolean functions after an explicit change of coordinates, not merely equivalent functions.
Thus their conjecture follows from the present result with no further input.

The proof technique is also meant to be reusable. The Hasse term is converted into a finite binary carry graph, the small recurrent components are evaluated directly, and the only large component is cancelled by transporting two conjugate projective frames along the equality graph. The few finite checks used in this step are parameter-free: they involve only the fixed six exponent vectors in the Hasse term and the fixed transition matrices displayed in Section~2. The scripts in Appendix~\ref{app:finite-checks} are part of the verification of these finite assertions, not numerical evidence for particular values of $q$. This gives a way of proving a congruence for a two-variable exponential sum without first determining its complete spectrum.

The paper is organized as follows.
Section~2 develops the finite-field reductions and proves the intrinsic Hasse congruence used for the outside spectrum.
Section~3 analyzes the Walsh values on $\fq$.
Section~4 treats the complementary regime $\fqq\setminus\fq$ in the noncubic case and proves Corollary~\ref{cor:distribution} and Corollary~\ref{cor:conj}.
Section~5 identifies the cyclotomic family of \cite{Xie2025} with the present permutation-inverse model.

\section{Finite-field reductions and the intrinsic Hasse input}

Throughout the sequel, we keep the assumptions that $e$ is even, $q=2^e$, $\Tr_q$ is the absolute trace from $\fq$ to $\ftwo$, and $\Tr_{q^2}$ the absolute trace from $\fqq$ to $\ftwo$.
We write
\[
\chi_q(x)=(-1)^{\Tr_q(x)}\qquad(x\in\fq)
\]
for the canonical additive character of $\fq$.
For $\beta\in\fqq$, write $\bar\beta=\beta^q$.
Let
\[
\mu_{q+1}=\set{z\in\fqq^*:z^{q+1}=1},
\qquad
\mu_{q+1}^*=\mu_{q+1}\setminus\set{1}.
\]
For $\alpha\in\fq^*$ and $\beta\in\fqq$, define
\[
P_{\alpha}(z)=\alpha(z^9+z^{-9}),
\qquad
Q_{\beta}(z)=\beta z^3+\bar\beta z^{-3}+(\beta+\bar\beta)(z+z^{-1}).
\]

\subsection{Finite-field reductions}

\begin{lem}\label{lem:basic-reduction}
For every $\alpha\in\fq^*$ and $\beta\in\fqq$ we have
\begin{equation}
W_{f_{\alpha}}(\beta)=
\sum_{(y,z)\in\fq\times\mu_{q+1}}\chi_q(P_{\alpha}(z)y^3+Q_{\beta}(z)y)-q.
\end{equation}
\end{lem}

\begin{proof}
Since $\sigma$ is a permutation of $\fqq$, we change variables
$x=\sigma(u)$ in the Walsh transform and then rename $u$ as $x$. Thus
\[
W_{f_{\alpha}}(\beta)
=
\sum_{x\in\fqq}
(-1)^{\Tr_{q^2}(\alpha x^3+\beta(x+x^d+x^{dq}))}.
\]
Let $d'=q^2-q+1$.
Since $\gcd(d',q^2-1)=1$, the monomial $x\mapsto x^{d'}$ permutes $\fqq$.
Also,
\[
 dd'=\frac{(q^2+q+1)(q^2-q+1)}3
    =1+\frac{q^2+2}{3}(q^2-1),
\]
and therefore $dd'\equiv 1\pmod{q^2-1}$.
After replacing the original variable by $x^{d'}$, the term $x^{dq}$ becomes
$x^{dd'q}=x^q$, because $dd'\equiv1\pmod{q^2-1}$. Moreover
\[
\Tr_{q^2}(\beta x^q)=\Tr_{q^2}((\beta x^q)^q)=\Tr_{q^2}(\bar\beta x).
\]
Hence
\begin{align*}
W_{f_{\alpha}}(\beta)
&=1+\sum_{x\in\fqq^*}(-1)^{\Tr_{q^2}(\alpha x^{3d'}+\beta x^{d'}+(\beta+\bar\beta)x)}.
\end{align*}
Note that the map $(y,z)\mapsto yz$ is a bijection from $\fq^*\times \mu_{q+1}$ to $\fqq^*$, so that we can write each nonzero element $x\in\fqq$ uniquely in the form $x=yz$ with
$y\in\fq^*$ and $z\in\mu_{q+1}$.
Since $(yz)^{d'}=yz^3$, we obtain
\begin{align*}
W_{f_{\alpha}}(\beta)-1
&=\sum_{(y,z)\in\fq^*\times\mu_{q+1}}
(-1)^{\Tr_{q^2}(\alpha z^9y^3+\beta z^3y+(\beta+\bar\beta)zy)}\\
&=\sum_{(y,z)\in\fq^*\times\mu_{q+1}}
(-1)^{\Tr_q(P_{\alpha}(z)y^3+Q_{\beta}(z)y)}.
\end{align*}
Adding the $y=0$ terms contributes $q+1$, and therefore
\[
W_{f_{\alpha}}(\beta)
=
\sum_{(y,z)\in\fq\times\mu_{q+1}}
\chi_q(P_{\alpha}(z)y^3+Q_{\beta}(z)y)-q,
\]
as desired. This proves Lemma~\ref{lem:basic-reduction}.
\end{proof}

\begin{lem}\label{lem:T-set}
Let
\[
T=\set{z+z^{-1}:z\in\mu_{q+1}^*}.
\]
Then
\[
T=\set{t\in\fq^*: \Tr_q(1/t)=1}.
\]
In particular, the map $z \mapsto z + z^{-1}$ from $\mu_{q+1}^*$ to $T$ is two-to-one, and therefore $|T| = q/2$.
\end{lem}

\begin{proof}
Let $t=z+z^{-1}$ with $z\in\mu_{q+1}^*$.
Then $t\in\fq^*$ and $z$ is a root of $X^2+tX+1$.
Moreover $z\notin\fq$: indeed, if $z\in\mu_{q+1}\cap\fq^*$, then $z^{q+1}=z^2=1$, whence $z=1$ in characteristic $2$, contrary to $z\in\mu_{q+1}^*$.
Thus $X^2+tX+1$ is irreducible over $\fq$.
By the Artin--Schreier criterion for quadratic polynomials in characteristic $2$, this is equivalent to
$\Tr_q(1/t^2)=1$, and hence to $\Tr_q(1/t)=1$.
Therefore $T\subseteq\set{t\in\fq^*: \Tr_q(1/t)=1}$.

Conversely, let $t\in\fq^*$ with $\Tr_q(1/t)=1$.
Then also $\Tr_q(1/t^2)=1$, so $X^2+tX+1$ is irreducible over $\fq$.
Let $z\in\fqq\setminus\fq$ be a root.
Its conjugate is $z^q$, and the Vieta relations give
\[
z+z^q=t,
\qquad
zz^q=1.
\]
Hence $z^{q+1}=1$, that is, $z\in\mu_{q+1}^*$, and then $t=z+z^{-1}\in T$.
This proves the claimed equality. For the map $z \mapsto z + z^{-1}$ from $\mu_{q+1}^*$ to $T$, $z$ and $z^{-1}$ have the same image and these are the only two preimages. It implies that this map is two-to-one, and the final claim is immediate. This finishes the proof of Lemma~\ref{lem:T-set}.
\end{proof}

\subsection{A denominator-free outside sum}

For $\delta\in\fq^*$ and $\Lambda\in\fq$ with $\Tr_q(\Lambda)=1$, put
\[
C_\Lambda=\Lambda^4+\Lambda^2+\Lambda
\]
and define the two-variable exponential sum
\begin{equation}\label{eq:def-F-delta-Lambda}
\mathcal F_{\delta,\Lambda}
=
\sum_{h,y\in\fq}
\chi_q\left(
 y(h^3+h^2+\Lambda h)
 +\delta y^3(h^8+h+C_\Lambda)
\right).
\end{equation}

\begin{prop}\label{prop:outside-reduction}
Let $\beta\in\fqq\setminus\fq$ and put $b=\beta+\bar\beta\in\fq^*$.
Choose $\lambda\in\fq$ with $\Tr_q(\lambda)=1$, let $\theta\in\fqq$ satisfy
\[
\theta^2+\theta=\lambda+1,
\qquad \theta^q=\theta+1,
\]
and write uniquely
\[
\beta=b(c+\theta),\qquad c\in\fq.
\]
Set
\[
\kappa=\alpha b^{-3},
\qquad
\Lambda=c^2+c+\lambda.
\]
Then $\Tr_q(\Lambda)=1$ and
\begin{equation}\label{eq:outside-F-delta-Lambda}
W_{f_\alpha}(\beta)+q=\mathcal F_{\kappa,\Lambda}.
\end{equation}
The auxiliary choice of $\lambda$ is only a coordinate choice: for every admissible
choice the displayed identity holds, and the later argument uses only the two
parameters $\kappa$ and $\Lambda$ with $\Tr_q(\Lambda)=1$.
\end{prop}

\begin{proof}
Since $e$ is even, $\Tr_q(1)=0$; hence $\Tr_q(\lambda+1)=1$, and
$X^2+X+\lambda+1$ is irreducible over $\fq$. Thus such a $\theta$ exists and
satisfies $\theta^q=\theta+1$.

For $x\in\fq$, put
\[
A_x=x^2+x+\lambda,
\qquad
z_x=\frac{x+\theta}{\sqrt{A_x+1}},
\qquad
s_x=z_x^3+z_x^{-3}.
\]
Here and below, for an element of $\fq$, the notation $r^{1/2}$ or $\sqrt r$
means the unique square root in $\fq$, namely $r^{q/2}$.
The elements $A_x$ and $A_x+1$ are nonzero, because $x^2+x$ has absolute
trace zero while $\lambda$ and $\lambda+1$ have absolute trace one. Thus all
square roots and inverses used in the definition of $z_x$ are legitimate.
Moreover,
\[
(x+\theta)(x+\theta^q)=A_x+1,
\]
so $z_x\in\mu_{q+1}^*$ and
\[
z_{x+1}=z_x^{-1},
\qquad
z_x+z_x^{-1}=(A_x+1)^{-1/2}.
\]
If $z_x=z_y$, then the displayed formula for $z_x+z_x^{-1}$ gives
$(A_x+1)^{-1/2}=(A_y+1)^{-1/2}$. Since the square-root map is an
automorphism of $\fq$, this implies $A_x=A_y$, so $y\in\{x,x+1\}$. The
second option would give $z_x=z_{x+1}=z_x^{-1}$ and hence $z_x=1$, impossible.
Therefore $x\mapsto z_x$ is injective, and hence a bijection from $\fq$ onto
$\mu_{q+1}^*$ because both sets have $q$ elements.

From $s_x=(z_x+z_x^{-1})^3+(z_x+z_x^{-1})$ one obtains
\[
s_x=A_x(A_x+1)^{-3/2},
\qquad
1+s_x^{-2}=A_x+A_x^{-1}+A_x^{-2}.
\]
In particular $s_x\ne0$ for every $x$, so the later change of variable by
$bs_x$ is a bijection of $\fq$.
Write
\[
\Phi(x)=A_x+A_x^{-1}+A_x^{-2}.
\]
We record the two elementary identities used in the sequel. First, in characteristic $2$,
\[
(z^3+z^{-3})^3+(z^3+z^{-3})=z^9+z^{-9},
\]
and hence $P_\alpha(z_x)=\alpha(s_x^3+s_x)$. Second, since
$\beta=b(c+\theta)$ and $\bar\beta=b(c+\theta+1)$, we have
\[
\begin{aligned}
 b^{-1}Q_\beta(z_x)
 &=c(z_x^3+z_x^{-3})+\theta z_x^3+(\theta+1)z_x^{-3}+z_x+z_x^{-1}.
\end{aligned}
\]
Writing $r=(A_x+1)^{1/2}$, so that $z_x=(x+\theta)/r$ and
$z_x^{-1}=(x+\theta+1)/r$, the last three terms have common denominator
$r^3$. Their numerator is
\begin{align*}
&\theta(x+\theta)^3+(\theta+1)(x+\theta+1)^3+(A_x+1) \\
&\qquad =x^3+x^2+\lambda x+1+x^2+x+\lambda+1 \\
&\qquad =x^3+(\lambda+1)x+\lambda=(x+1)(x^2+x+\lambda)=(x+1)A_x .
\end{align*}
Here only $\theta^2+\theta=\lambda+1$ and $A_x=x^2+x+\lambda$ are used; the
first equality is a direct expansion in characteristic two. Since $s_x=A_x/r^3$,
this gives
\[
P_\alpha(z_x)=\alpha(s_x^3+s_x),
\qquad
Q_\beta(z_x)=b(c+x+1)s_x.
\]
By Lemma~\ref{lem:basic-reduction}, the term $z=1$ contributes
$\sum_{y\in\fq}\chi_q(by)=0$. Using the bijection $x\mapsto z_x$ and then the
change of variable $u=bs_xy$, we get
\begin{align*}
W_{f_\alpha}(\beta)+q
&=\sum_{x,y\in\fq}
\chi_q\bigl(\alpha(s_x^3+s_x)y^3+b(c+x+1)s_xy\bigr)\\
&=\sum_{x,u\in\fq}
\chi_q\bigl(\kappa\Phi(x)u^3+(c+x+1)u\bigr).
\end{align*}
Now set
\[
h=x+c+1,
\qquad
\Lambda=c^2+c+\lambda.
\]
Then $\Tr_q(\Lambda)=\Tr_q(\lambda)=1$ and
\[
A_x=x^2+x+\lambda=h^2+h+\Lambda.
\]
Let $A=h^2+h+\Lambda$. Since $\Tr_q(\Lambda)=1$, $A\ne0$ for every $h\in\fq$:
otherwise $h^2+h=\Lambda$ would force $\Tr_q(\Lambda)=0$. Hence the change of
variable $u=Ay$ is a bijection on $\fq$ for each fixed $h$. Then
\[
\Phi(x)u^3=(A+A^{-1}+A^{-2})A^3y^3=(A^4+A^2+A)y^3,
\]
where
\[
A^4+A^2+A=h^8+h+\Lambda^4+\Lambda^2+\Lambda=h^8+h+C_\Lambda.
\]
Also
\[
(c+x+1)u=hAy=y(h^3+h^2+\Lambda h).
\]
Substitution gives exactly \eqref{eq:def-F-delta-Lambda} with $\delta=\kappa$.
This proves \eqref{eq:outside-F-delta-Lambda}.
\end{proof}

\subsection{The intrinsic Hasse congruence}

The following congruence is the key input for the outside sums in the noncubic case.

\begin{prop}\label{prop:intrinsic-hasse}
Let $q=2^e$ with $e$ even, and put $N=q-1$.
Let $\delta\in\fq^*$ and let $\Lambda\in\fq$ satisfy $\Tr_q(\Lambda)=1$.
With $C_\Lambda=\Lambda^4+\Lambda^2+\Lambda$, one has
\begin{equation}\label{eq:intrinsic-hasse-congruence}
\mathcal F_{\delta,\Lambda}
\equiv
q\left(1+\delta^{N/3}+\delta^{2N/3}\right)
\pmod {2q}.
\end{equation}
Here $1+\delta^{N/3}+\delta^{2N/3}\in\ftwo$, and it is identified with the integer $0$ or $1$ in the congruence.
\end{prop}

The proof occupies the rest of this subsection. The computer-assisted input is
limited to four parameter-free finite certificates, all recorded in
Appendix~\ref{app:finite-checks}. The first certificate verifies the potential
inequality and lists the equality-cycle components used in
Lemma~\ref{lem:minimal-carry-graph}. The second rebuilds, from the same equality
graph, the weighted transition blocks displayed before
Lemma~\ref{lem:two-state-frame}. The third verifies the trace-zero algebra used
in Lemma~\ref{lem:G7-trace-zero}. The final SageMath check verifies the rational
projective identities used in Lemma~\ref{lem:G7-kernel-frame}. None of these
checks depends on $q$, $e$, $\delta$, or $\Lambda$; they are finite certificates
for the stated combinatorial and matrix assertions, not numerical evidence for
special finite fields.
For the polynomial in \eqref{eq:def-F-delta-Lambda}, write
\[
E_{\delta,\Lambda}(h,y)
=h^3y+h^2y+\Lambda hy+\delta h^8y^3+\delta hy^3+\delta C_\Lambda y^3.
\]
Since $\Tr_q(\Lambda)=1$, we have $\Lambda\ne0$ and also $C_\Lambda\ne0$.
Indeed, if $C_\Lambda=0$, then $\Lambda^3+\Lambda+1=0$. A root of
$X^3+X+1$ lying in $\mathbb F_{2^e}$ can occur only when $3\mid e$, and then it
lies in $\mathbb F_8$. Its trace from $\mathbb F_8$ to $\ftwo$ is zero, because
the coefficient of $X^2$ in its minimal polynomial is zero. Hence its absolute
trace from $\mathbb F_{2^e}$ to $\ftwo$ is $(e/3)\cdot0=0$, contradicting
$\Tr_q(\Lambda)=1$.

Let
\begin{align}\label{degreeofE}
d_1=(3,1),\ d_2=(2,1),\ d_3=(1,1),\ d_4=(8,3),\ d_5=(1,3),\ d_6=(0,3)
\end{align}
be the degree vectors of terms in $E_{\delta,\Lambda}$
with coefficients
\[
a_1=1,\quad a_2=1,\quad a_3=\Lambda,
\qquad
 a_4=\delta,\quad a_5=\delta,\quad a_6=\delta C_\Lambda.
\]
For $U=(u_1,\ldots,u_6)$ with $0\le u_i\le N$, define
\[
w(U)=\sum_{i=1}^6 s_2(u_i),
\]
where $s_2$ denotes the binary digit sum. Let $\mathcal M_e$ be the set of all such $U$ satisfying
\begin{equation}\label{eq:modular-system}
\sum_{i=1}^6u_id_i\equiv(0,0)\pmod N,
\end{equation}
with both actual coordinate sums positive and with $w(U)=e$.

The following elementary first-term lemma will be used in the proof of Proposition \ref{prop:intrinsic-hasse}. It is the binary specialization of the usual Gauss-sum--Stickelberger
calculation; the Gauss-sum congruence used below is the standard
Stickelberger congruence, for instance in the form of
\cite[Theorem~11.2.1]{BEW1998}. For the classical cyclotomic formulation of
Stickelberger's theorem, see also \cite[Theorem~6.10]{Washington1997}. We spell
out the specialization in order to fix the endpoint conventions. In particular,
the two exponents $0$ and $N=q-1$ are not identified: $T^N$ is the indicator of
$T\ne0$ on $\fq$, whereas $T^0$ is the constant function.

\begin{lem}\label{lem:binary-hasse-term}
Let $q=2^e$, $N=q-1$, and let
\[
H(X,Y)=\sum_{i=1}^r a_iX^{D_{i,1}}Y^{D_{i,2}}\in\fq[X,Y]
\]
with all $a_i\ne0$. For $U=(u_1,\ldots,u_r)$ with $0\le u_i\le N$, put
\[
M_1(U)=\sum_i u_iD_{i,1},\qquad
M_2(U)=\sum_i u_iD_{i,2},\qquad
w(U)=\sum_i s_2(u_i).
\]
Assume that every $U$ for which $M_1(U)$ and $M_2(U)$ are positive multiples of
$N$ satisfies $w(U)\ge e$. Then
\[
\sum_{X,Y\in\fq}\chi_q(H(X,Y))
\equiv
q\sum_{\substack{0\le u_i\le N\\ M_1(U),M_2(U)>0\\ M_1(U)\equiv M_2(U)\equiv0\ (N)\\ w(U)=e}}
\prod_{i=1}^r a_i^{u_i}
\pmod{2q}.
\]
The sum on the right is first interpreted in $\fq$; its residue is
Frobenius-fixed in the proof below, hence lies in $\ftwo$, and is then
identified with the integer $0$ or $1$.
\end{lem}

\begin{proof}
Work in the ring $\mathcal O$ of integers of the unramified $2$-adic extension
with residue field $\fq$, and write hats for Teichm{\"u}ller lifts. The function
$t\mapsto\chi_q(t)$ on $\fq$ has a unique interpolation on Teichm{\"u}ller
representatives of the form
\[
\chi_q(t)=\sum_{u=0}^{N} b_u\widehat t^{\,u},
\qquad \widehat0^0=1,
\]
where $\widehat 0^{\,u}=0$ for $u>0$. Uniqueness follows because the
multiplicative characters $t\mapsto\widehat t^{\,u}$, $0\le u<N$, are linearly
independent on $\fq^*$, and the additional function $t\mapsto\widehat t^{\,N}$
separates $0$ from $\fq^*$.
The coefficients are
\[
b_0=1,
\qquad
b_N=-\frac qN,
\qquad
b_u=\frac1N\sum_{t\in\fq^*}\chi_q(t)\widehat t^{-u}\quad(1\le u\le N-1).
\]
The coefficient $b_N$ is obtained from
$N^{-1}\sum_{t\in\fq^*}(\chi_q(t)-1)$ and
$\sum_{t\in\fq^*}\chi_q(t)=-1$. For $1\le u\le N-1$, the classical
Stickelberger congruence for binary Gauss sums
\cite[Theorem~11.2.1]{BEW1998} gives
\[
2^{-s_2(u)}b_u\equiv1\pmod {2\mathcal O}.
\]
Together with $b_0=1$ and $2^{-e}b_N=-1/N\equiv1\pmod {2\mathcal O}$, this says
\begin{equation}\label{eq:bu-first-unit}
2^{-s_2(u)}b_u\equiv1\pmod {2\mathcal O}\qquad(0\le u\le N),
\end{equation}
where $s_2(0)=0$ and $s_2(N)=e$. Thus the possible signs and the usual binary
digit-factorial denominators disappear after reduction modulo $2$.

Using the additivity of $\chi_q$, expand
\[
\chi_q(H(X,Y))=\prod_{i=1}^r\chi_q(a_iX^{D_{i,1}}Y^{D_{i,2}})
\]
by the above interpolation. Summing term by term gives
\[
\sum_{X,Y\in\fq}\chi_q(H(X,Y))
=
\sum_{0\le u_i\le N}
\left(\prod_i b_{u_i}\widehat a_i^{\,u_i}\right)
\left(\sum_{X\in\fq}\widehat X^{M_1(U)}\right)
\left(\sum_{Y\in\fq}\widehat Y^{M_2(U)}\right).
\]
For a nonnegative integer $M$, one has
\[
\sum_{X\in\fq}\widehat X^{M}= 
\begin{cases}
q,& M=0,\\
N,& M>0\text{ and }N\mid M,\\
0,& M>0\text{ and }N\nmid M.
\end{cases}
\]
Consequently a nonzero contribution modulo $2q$ must have
\[
M_1(U)\equiv M_2(U)\equiv0\pmod N.
\]
If one of $M_1(U),M_2(U)$ is zero, then either $U=0$, giving the factor
$q^2$, which is congruent to $0$ modulo $2q$ since $e\ge1$, or $U\ne0$ and the
coefficient already has positive $2$-adic order; in the latter case the extra
factor $q$ makes the term divisible by $2q$. Thus, modulo $2q$, only the terms
with both actual coordinate sums positive can remain.

For such terms the two coordinate sums contribute the odd factor $N^2$, which is
$1$ modulo $2$. By \eqref{eq:bu-first-unit}, a term of weight $w(U)$ is congruent
to
\[
2^{w(U)}\prod_i a_i^{u_i}
\]
up to a factor congruent to $1$ modulo $2$. The assumed lower bound
$w(U)\ge e$ eliminates all weights below $e$, and all weights above $e$ are
zero modulo $2q=2^{e+1}$. The surviving terms are exactly those with $w(U)=e$,
and reducing their unit factors modulo $2$ gives the displayed congruence in
$\mathcal O/(2q)$. Since the left-hand side is a rational integer, applying the Frobenius
automorphism to the congruence cannot change its residue modulo $2q$. Therefore
the coefficient of $q$ that remains after division by $q$ is fixed in the residue
field modulo $2$, hence lies in $\ftwo$. This is the element identified with the
integer $0$ or $1$ in the statement.
\end{proof}

\begin{rem}
Lemma~\ref{lem:binary-hasse-term} may be viewed as the first Hasse term behind
Ax--Katz type estimates. General bounds and density formalisms are developed,
for example, in \cite{Moreno2004} and \cite{Blache2009}. The present proof uses
only the explicit binary first term above and not those general frameworks.
\end{rem}

\begin{lem}\label{lem:first-hasse-term}
With the notation above,
\begin{equation}\label{eq:first-hasse-term}
\mathcal F_{\delta,\Lambda}
\equiv
q\sum_{U\in\mathcal M_e}
\Lambda^{u_3}\delta^{u_4+u_5+u_6}C_\Lambda^{u_6}
\pmod{2q}.
\end{equation}
\end{lem}

\begin{proof}
Apply Lemma~\ref{lem:binary-hasse-term} to $E_{\delta,\Lambda}$. The lower bound
$w(U)\ge e$ required in that lemma is proved below in
Lemma~\ref{lem:minimal-carry-graph}. The remaining weight-$e$ solutions are
precisely the elements of $\mathcal M_e$. Since the coefficients of
$E_{\delta,\Lambda}$ are $a_1,\ldots,a_6$ as displayed above,
\[
\prod_{i=1}^6a_i^{u_i}
=\Lambda^{u_3}\delta^{u_4+u_5+u_6}C_\Lambda^{u_6}.
\]
This proves \eqref{eq:first-hasse-term}.
\end{proof}

We next describe the minimal solutions of \eqref{eq:modular-system}. Write
\[
u_i=\sum_{j=0}^{e-1}u_{i,j}2^j,\qquad u_{i,j}\in\{0,1\},
\]
and set
$m_j=\{i:u_{i,j}=1\}\subseteq\{1,\ldots,6\}$.
Let
\[
d(m_j)=\sum_{i\in m_j}d_i,
\]
where each $d_i$ is given as in \eqref{degreeofE}. A solution of \eqref{eq:modular-system} gives an edge-labelled cyclic carry
path as follows. Indices on the labels are read modulo $e$. Since
$\sum_j2^jd(m_j)$ is divisible by $N$ coordinatewise, so is every cyclic shift
of this sum. Define
\[
\phi_j=\frac1N\sum_{r=0}^{e-1}2^r d(m_{j+r})\qquad(0\le j<e).
\]
The coordinatewise divisibility follows because
\[
\sum_{r=0}^{e-1}2^r d(m_{j+r})\equiv
2^{-j}\sum_{r=0}^{e-1}2^r d(m_r)\pmod N.
\]
Thus each $\phi_j$ is an integral vector. If both actual coordinate sums in
\eqref{eq:modular-system} are positive, then both coordinates of every
$\phi_j$ are positive. Moreover,
\[
2\phi_{j+1}-\phi_j
=\frac1N\left(2\sum_{r=0}^{e-1}2^r d(m_{j+1+r})-
\sum_{r=0}^{e-1}2^r d(m_{j+r})\right)
=d(m_j),
\]
because the middle numerator equals $(2^e-1)d(m_j)=Nd(m_j)$. Hence a solution
with positive actual coordinate sums determines an edge-labelled cyclic carry
path
\[
(\phi_j,m_j)_{j=0}^{e-1},\qquad \phi_e=\phi_0,
\]
with
\begin{equation}\label{eq:carry-equation}
d(m_j)=2\phi_{j+1}-\phi_j.
\end{equation}
Conversely, an edge-labelled cyclic carry path $(\phi_j,m_j)_{j=0}^{e-1},\ \phi_e=\phi_0$
satisfying \eqref{eq:carry-equation} determines a solution uniquely: set \(u_{i,j}=1\) if and only if \(i\in m_j\), and put
\[
u_i=\sum_{j=0}^{e-1}u_{i,j}2^j .
\]
Then \(0\le u_i\le N\). Moreover,
\[
\sum_{i=1}^6u_id_i
=
\sum_{j=0}^{e-1}2^j d(m_j)
=
\sum_{j=0}^{e-1}2^j(2\phi_{j+1}-\phi_j)
=
2^e\phi_e-\phi_0
=
N\phi_0.
\]
Hence
\[
\sum_{i=1}^6u_id_i\equiv (0,0)\pmod N.
\]
Thus the constructed \(U=(u_1,\ldots,u_6)\) is indeed a solution of
\eqref{eq:modular-system}. Its weight is
\[
w(U)=\sum_{i=1}^6s_2(u_i)=\sum_{j=0}^{e-1}|m_j|.
\]
Here the labels \(m_j\) are part of the path data; the vertex sequence
\((\phi_j)\) alone need not determine them uniquely.

\begin{lem}\label{lem:minimal-carry-graph}
Every solution of \eqref{eq:modular-system} whose two actual coordinate sums are positive has $w(U)\ge e$. In the equality case, the directed cycle part of the carry graph consists of exactly the following three components:
\[
\mathcal C_0=\{(3,1)\},\quad\mathcal C_2=\{(3,2),(6,4)\},
\]
and
\[
\begin{aligned}
\mathcal C_1=\{&(1,1),(1,2),(1,3),(2,1),(2,2),(2,3),(3,3),
(4,2),(4,3),(4,4),\\
&(5,2),(5,3),(5,4),(6,3),(7,3),(7,4),(7,5),(8,3),(8,4),(8,5)\}.
\end{aligned}
\]
Moreover, $\mathcal C_0$ contributes $1$, and $\mathcal C_2$ contributes
$\delta^{N/3}+\delta^{2N/3}$.
\end{lem}

\begin{proof}
Let $d(m)=(d_1(m),d_2(m))\in\mathbb{Z}^2$. Since
\[
0\le d_1(m)\le 3+2+1+8+1=15,
\qquad
0\le d_2(m)\le 1+1+1+3+3+3=12,
\]
every cyclic carry $\phi_j=(\phi_{j,1},\phi_{j,2})$ satisfying \eqref{eq:carry-equation} lies in the rectangle
\begin{equation}\label{eq:carry-rectangle}
1\le \phi_1\le 15,
\qquad
1\le \phi_2\le 12.
\end{equation}
Indeed, multiplying \eqref{eq:carry-equation} by $2^j$ and summing over
$j=0,\ldots,e-1$ gives
\[
\sum_{j=0}^{e-1}2^jd(m_j)=2^e\phi_e-\phi_0=N\phi_0.
\]
The left-hand side is the vector of actual coordinate sums. Hence the two
coordinates of $\phi_0$ are positive. Since all digits $d_i(m_j)$ are
nonnegative, no later carry coordinate can be zero: if
$\phi_{j+1,r}=0$, then $d_r(m_j)=2\phi_{j+1,r}-\phi_{j,r}$ forces
$\phi_{j,r}=d_r(m_j)=0$, and propagating this implication backwards around the
cycle gives $\phi_{0,r}=0$, a contradiction. If $M_r$ is the maximum of the
$r$-th carry coordinate along the cycle, then at a place where this maximum is
attained,
\[
M_r\le \frac{M_r+D_r}{2},
\qquad D_1=15,\quad D_2=12,
\]
which gives $M_r\le D_r$.

On the rectangle \eqref{eq:carry-rectangle} use the potential function $\nu$ given by the following table; rows are indexed by the second coordinate and columns by the first coordinate:
\[
\begin{array}{c|ccccccccccccccc}
y\backslash x
&1&2&3&4&5&6&7&8&9&10&11&12&13&14&15\\ \hline
1&3&3&2&1&1&1&1&0&0&0&0&0&0&0&0\\
2&3&2&2&2&3&1&1&0&0&0&0&0&0&0&0\\
3&2&2&2&2&2&2&2&2&0&0&0&0&0&0&0\\
4&2&1&2&1&2&1&2&1&1&0&0&0&0&0&0\\
5&1&1&1&1&1&1&1&1&0&0&0&0&0&0&0\\
6&1&1&1&1&1&1&1&0&0&0&0&0&0&0&0\\
7&0&0&0&0&0&0&0&0&0&0&0&0&0&0&0\\
8&0&0&0&0&0&0&0&0&0&0&0&0&0&0&0\\
9&0&0&0&0&0&0&0&0&0&0&0&0&0&0&0\\
10&0&0&0&0&0&0&0&0&0&0&0&0&0&0&0\\
11&0&0&0&0&0&0&0&0&0&0&0&0&0&0&0\\
12&0&0&0&0&0&0&0&0&0&0&0&0&0&0&0
\end{array}
\]
The only assertion about this table is the following finite statement. The
potential has no intrinsic role beyond certifying this inequality; no optimality
or uniqueness of $\nu$ is required. For each of the $64$ subsets
$m\subseteq\{1,\ldots,6\}$ and each carry $\phi$ in \eqref{eq:carry-rectangle} for which
\[
\phi'=\frac{\phi+d(m)}2
\]
has integral coordinates and also lies in \eqref{eq:carry-rectangle}, one has
\begin{equation}\label{eq:potential-ineq}
|m|-1-\nu(\phi)+\nu(\phi')\ge0.
\end{equation}
The table was found by a finite shortest-potential search, but only the
following certified statement is used. It is a finite check over
$64\cdot 15\cdot12$ cases; Appendix~\ref{app:finite-checks} gives a complete
script which verifies \eqref{eq:potential-ineq} and lists exactly the directed
cycle components of the equality graph. Summing \eqref{eq:potential-ineq} around a closed carry path gives
\[
w(U)-e\ge0.
\]
Equality can occur only when every edge of the closed path is an equality edge in \eqref{eq:potential-ineq}. The directed cycles in this equality graph are precisely the three components $\mathcal C_0,\mathcal C_1,\mathcal C_2$ displayed in the statement. All other equality vertices, if any, are transient and therefore cannot occur in a closed minimal carry path.

The component $\mathcal C_0$ has the single loop
\[
(3,1)\xrightarrow{\{1\}}(3,1),
\]
which gives $u_1=N$ and all other $u_i=0$; hence its contribution is $1$.
The component $\mathcal C_2$ is the two-cycle
\[
(3,2)\xrightarrow{\{4,5\}}(6,4),
\qquad
(6,4)\xrightarrow{\varnothing}(3,2).
\]
Since $e$ is even, the two possible cyclic phases contribute respectively
\[
\delta^{2(1+2^2+\cdots+2^{e-2})}=\delta^{2N/3}
\]
and
\[
\delta^{2(2+2^3+\cdots+2^{e-1})}=\delta^{4N/3}=\delta^{N/3},
\]
because $\delta^N=1$ in $\fq^*$ and $4N/3\equiv N/3\pmod N$.
This proves the lemma.
\end{proof}

It remains to show that the large component $\mathcal C_1$ contributes zero.
The idea is to encode its closed walks by transfer matrices and then to use the
Artin--Schreier boundary relation below to exchange two conjugate moving
projective lines after one full period. Thus a total return matrix becomes
anti-diagonal in a suitable basis, and its trace, which is exactly the closed-walk
contribution, is zero. The next three lemmas make this cancellation explicit.
Let $\eta\in\fqq$ satisfy
\[
\eta^2+\eta=\Lambda.
\]
Since $\Tr_q(\Lambda)=1$, one has
\begin{equation}\label{eq:eta-boundary}
\eta^q+\eta=1.
\end{equation}
Put $\eta_j=\eta^{2^j}$, and set
\[
L_j=\Lambda^{2^j}=\eta_{j+1}+\eta_j,
\qquad
C_j=C_\Lambda^{2^j}=\eta_{j+3}+\eta_j,
\qquad
D_j=\delta^{2^j}.
\]
The following partition is induced by the equality graph together with the
weighted blocks verified in Appendix~\ref{app:finite-checks}. The six two-state
groups are exactly the parts on which the moving two-line frame below is
transported directly. The remaining eight-state group is the only region through
which paths have to be collapsed into excursion kernels $K_{a,b}$.
We split $\mathcal C_1$ into seven groups:
\[
G_1=\{(1,1),(2,1)\},\quad
G_2=\{(1,2),(5,2)\},\quad
G_3=\{(2,2),(4,2)\},
\]
\[
G_4=\{(4,4),(8,4)\},\quad
G_5=\{(5,4),(7,4)\},\quad
G_6=\{(7,5),(8,5)\},
\]
\[
G_7=\{(1,3),(2,3),(3,3),(4,3),(5,3),(6,3),(7,3),(8,3)\}.
\]
The six two-state groups have shifts
\[
s_1=0,
\quad s_2=2,
\quad s_3=1,
\quad s_4=2,
\quad s_5=1,
\quad s_6=0.
\]
For $1\le i\le 6$ define two moving projective vectors
\begin{equation}
p_j^{(i)}=\binom{\eta_{j+s_i}^{-1}}1,
\qquad
\bar p_j^{(i)}=\binom{(\eta_{j+s_i}+1)^{-1}}1.
\end{equation}
The denominators are nonzero because \eqref{eq:eta-boundary} implies $\eta_j\notin\{0,1\}$ for every $j$.
Throughout this part, $\langle v_1,\ldots,v_r\rangle$ denotes the
$\fqq$-linear span of the displayed vectors. In particular, for a nonzero
vector $v$, the notation $\langle v\rangle$ denotes the one-dimensional
subspace spanned by $v$, or equivalently the corresponding projective line.
All transition matrices in the large component are written with rows indexed by
the source group and columns indexed by the target group. Thus, for the purpose
of closed-walk weights, products are taken from left to right; when these
matrices act on column vectors, they are pullback maps from the target coordinate
space to the source coordinate space. For example, the edge
$(1,1)\xrightarrow{\{3\}}(1,1)$ contributes the entry $L_j$ in row $(1,1)$ and
column $(1,1)$ of $A_j$, while
$(1,1)\xrightarrow{\{1\}}(2,1)$ contributes the entry $1$ in row $(1,1)$ and
column $(2,1)$. The displayed matrices below give all equality-graph transitions
in the large component; Appendix~\ref{app:finite-checks} verifies the edge list
and the weights directly. Some blocks contain scalar factors such as $D_j$ or
$D_j^2$. These factors are kept in the transfer-matrix weights, but they do not
affect the projective-line containments used below; multiplying a block by a
nonzero scalar preserves exactly the same projective lines.

\begin{lem}\label{lem:two-state-frame}
Every two-state transition in $\mathcal C_1$ maps the two moving lines into the corresponding two moving lines. More precisely, if $M_j:G_i\to G_k$ is one of the two-state blocks below, viewed as a pullback matrix from the target columns to the source rows, then
\[
M_jp_{j+1}^{(k)}\in\langle p_j^{(i)}\rangle,
\qquad
M_j\bar p_{j+1}^{(k)}\in\langle \bar p_j^{(i)}\rangle.
\]
\end{lem}

\begin{proof}
The required two-state blocks are
\[
G_1\to G_1:
\quad
A_j=\begin{pmatrix}L_j&1\\0&1\end{pmatrix},
\]
\[
G_1\to G_2,\\ G_6\to G_4:
\quad
B_j=D_j\begin{pmatrix}1&0\\ C_j&1\end{pmatrix},
\]
\[
G_2\to G_3:
\quad L_jI_2,
\qquad
G_2\to G_5,\ G_5\to G_6:
\quad D_j^2I_2,
\]
\[
G_3\to G_1,\ G_4\to G_3:
\quad I_2,
\]
\[
G_5\to G_4:
\quad
E_j=D_j\begin{pmatrix}1+C_j&1\\ L_jC_j&L_j\end{pmatrix}.
\]
Scalar factors do not affect projective lines, and the factors $D_j,D_j^2$
are nonzero because $\delta\in\fq^*$. Thus the projective checks may be made
after removing those scalars. The identities
\[
L_j=\eta_{j+1}+\eta_j,
\qquad
C_j=\eta_{j+3}+\eta_j
\]
are enough for all checks.
For example,
\[
A_j\binom{\eta_{j+1}^{-1}}1
=\binom{L_j\eta_{j+1}^{-1}+1}1
=\binom{\eta_j^{-1}}1,
\]
and similarly with $\eta_r$ replaced by $\eta_r+1$.
For $B_j$, ignoring $D_j$,
\[
\begin{pmatrix}1&0\\C_j&1\end{pmatrix}
\binom{\eta_{j+3}^{-1}}1
=\binom{\eta_{j+3}^{-1}}{(C_j+\eta_{j+3})\eta_{j+3}^{-1}}
=\eta_j\eta_{j+3}^{-1}\binom{\eta_j^{-1}}1.
\]
Again the barred case follows from $C_j+\eta_{j+3}+1=\eta_j+1$.
For $E_j$, ignoring $D_j$ and applying it to $\binom{\eta_{j+3}^{-1}}1$, the ratio of the two coordinates is
\[
\frac{(1+C_j)\eta_{j+3}^{-1}+1}
{L_jC_j\eta_{j+3}^{-1}+L_j}
=\eta_{j+1}^{-1},
\]
because $C_j+\eta_{j+3}=\eta_j$, $L_j=\eta_{j+1}+\eta_j$, and $\eta_{j+1}=\eta_j^2$.
The denominator in this ratio is nonzero; indeed it equals
$(\eta_j+1)\eta_j^{-6}$ after rewriting everything in terms of $\eta_j$,
and $\eta_j\notin\{0,1\}$. The barred case is identical after replacing
every $\eta_r$ by $\eta_r+1$. The remaining blocks are scalar multiples of
the identity or identities, and the chosen shifts make the source and target
frames agree. This proves the lemma.
\end{proof}

We now handle paths through the eight-state group $G_7$. Order the states of $G_7$ as
\[
(1,3),(2,3),(3,3),(4,3),(5,3),(6,3),(7,3),(8,3).
\]
The internal transition is $D_jR_j$, where $R_j=R(C_j)=P+C_jQ$ and
\[
P=\begin{pmatrix}
1&0&0&0&0&0&0&0\\
0&0&0&0&1&0&0&0\\
0&1&0&0&0&0&0&0\\
0&0&0&0&0&1&0&0\\
0&0&1&0&0&0&0&0\\
0&0&0&0&0&0&1&0\\
0&0&0&1&0&0&0&0\\
0&0&0&0&0&0&0&1
\end{pmatrix},
\quad
Q=\begin{pmatrix}
0&0&0&0&0&0&0&0\\
1&0&0&0&0&0&0&0\\
0&0&0&0&0&0&0&0\\
0&1&0&0&0&0&0&0\\
0&0&0&0&0&0&0&0\\
0&0&1&0&0&0&0&0\\
0&0&0&0&0&0&0&0\\
0&0&0&1&0&0&0&0
\end{pmatrix}.
\]
The transition from $G_2$ to $G_7$ has an overall scalar factor $D_j$; after
removing this nonzero scalar it is
\[
S_j=\begin{pmatrix}
L_jC_j&1+C_j&0&0&L_j&1&0&0\\
0&0&L_jC_j&1+C_j&0&0&L_j&1
\end{pmatrix},
\]
and the transition from $G_7$ to $G_3$ is
\[
T_j=\begin{pmatrix}
1&0\\
1&0\\
L_j&0\\
0&0\\
0&1\\
0&1\\
0&L_j\\
0&0
\end{pmatrix}.
\]
A maximal excursion through $G_7$ therefore gives, after deleting the nonzero
overall scalar $D_aD_{a+1}\cdots D_{b-1}$, an effective two-state kernel
\begin{equation}\label{eq:G7-kernel}
K_{a,b}=S_aR_{a+1}R_{a+2}\cdots R_{b-1}T_b
\qquad(a<b).
\end{equation}
The omitted scalar is still present in the closed-walk weight, but it is
irrelevant for the projective containment asserted in Lemma~\ref{lem:G7-kernel-frame}.

\begin{lem}\label{lem:G7-kernel-frame}
For every $a<b$,
\[
K_{a,b}\binom{\eta_{b+2}^{-1}}1
\in
\left\langle\binom{\eta_{a+2}^{-1}}1\right\rangle,
\]
and
\[
K_{a,b}\binom{(\eta_{b+2}+1)^{-1}}1
\in
\left\langle\binom{(\eta_{a+2}+1)^{-1}}1\right\rangle.
\]
Equivalently,
\[
K_{a,b}p_{b+1}^{(3)}\in\langle p_a^{(2)}\rangle,
\qquad
K_{a,b}\bar p_{b+1}^{(3)}\in\langle\bar p_a^{(2)}\rangle.
\]
\end{lem}

\begin{proof}
We prove the first assertion for the scalar-free kernel \eqref{eq:G7-kernel};
restoring the omitted nonzero scalar only multiplies the final vector and hence
does not change the projective line. The barred assertion is obtained by replacing every
$\eta_r$ by $\eta_r+1$. This replacement leaves $L_j$ and $C_j$ unchanged and is
compatible with Frobenius, since $(\eta_r+1)^2=\eta_{r+1}+1$. The identities are
formal Frobenius translates and are valid for arbitrary integers in the displayed
ranges. In the symbolic check one
first clears denominators; the only possible denominators are powers of $\eta_j$
and $\eta_j+1$, both nonzero by \eqref{eq:eta-boundary}. Appendix~\ref{app:finite-checks}
gives a reproducible SageMath check over $\ftwo(t)$ for the required rational
identities.
Put $x_r=\eta_r^{-1}$ and
\[
u_j=T_j\binom{x_{j+2}}1
=\begin{pmatrix}
x_{j+2}\\x_{j+2}\\L_jx_{j+2}\\0\\1\\1\\L_j\\0
\end{pmatrix}.
\]
Let
\[
V_j=\langle u_j,
R_ju_{j+1},
R_jR_{j+1}u_{j+2}\rangle.
\]
A direct multiplication, using only $\eta_{r+1}=\eta_r^2$, gives
\begin{equation}\label{eq:G7-invariance}
R_jV_{j+1}\subseteq V_j.
\end{equation}
For hand verification, the only non-immediate vector is $R_jR_{j+1}R_{j+2}u_{j+3}$; after writing all entries in terms of $\eta_j$, it equals
\[
\frac{\eta_j^6+\eta_j^5+\eta_j^4+\eta_j^3+\eta_j^2+\eta_j+1}{\eta_j^{28}}u_j
+\frac{\eta_j^5+\eta_j^4+\eta_j+1}{\eta_j^{23}}R_ju_{j+1}
+\frac{\eta_j+1}{\eta_j^{13}}R_jR_{j+1}u_{j+2}.
\]
Now set
\[
\alpha_a=(1,\eta_{a+2}^{-1})S_a.
\]
Another direct substitution gives
\[
\alpha_a u_{a+1}=0,
\qquad
\alpha_aR_{a+1}u_{a+2}=0,
\qquad
\alpha_aR_{a+1}R_{a+2}u_{a+3}=0.
\]
Thus $\alpha_aV_{a+1}=0$. By \eqref{eq:G7-invariance},
\[
R_{a+1}R_{a+2}\cdots R_{b-1}u_b\in V_{a+1}.
\]
Therefore
\[
(1,\eta_{a+2}^{-1})K_{a,b}\binom{\eta_{b+2}^{-1}}1=0,
\]
which is equivalent to the claimed projective containment.
\end{proof}

\begin{lem}\label{lem:G7-trace-zero}
For arbitrary $C_0,\ldots,C_{r-1}$ in any field of characteristic $2$,
\[
\operatorname{tr}\bigl(R(C_0)R(C_1)\cdots R(C_{r-1})\bigr)=0.
\]
\end{lem}

\begin{proof}
Let $K$ be the ground field. The $K$-linear span of the following sixteen
$0$-$1$ matrices is closed under right multiplication by $P$ and by $Q$:
\[
I_8,
P,
Q,
P^2,
PQ,
QP,
Q^2,
P^2Q,
PQP,
PQ^2,
QP^2,
QPQ,
Q^2P,
Q^3,
P^2QP,
P^2Q^2.
\]
This closure is one of the finite matrix checks recorded in
Appendix~\ref{app:finite-checks}. Each listed matrix has trace zero as an
element of a characteristic-two field. In particular, $\operatorname{tr}(I_8)=8=0$
in characteristic $2$. Since each $R(C)=P+CQ$ belongs to this span and the span
is stable under right multiplication by $P$ and by $Q$, every product of the
$R(C_i)$ belongs to the same trace-zero linear space, and the assertion follows.
\end{proof}

\begin{lem}\label{lem:large-component-zero}
The total contribution of the large carry component $\mathcal C_1$ to the coefficient of $q$ on the right-hand side of \eqref{eq:first-hasse-term} is zero.
\end{lem}

\begin{proof}
The coefficient of $q$ in \eqref{eq:first-hasse-term} is computed by transfer
matrices. For a fixed equality component, form at the $j$-th digit the block
matrix whose entries are the coefficient weights of equality edges from time
$j$ to time $j+1$. The trace of the product of these digit matrices is exactly
the sum of the weights of all closed equality-graph walks in the component: an
expanded diagonal entry records the successive states of one closed walk, and
every closed walk is obtained once from its chosen base state.

We first take the base point in one of $G_1,\ldots,G_6$. On the time interval
$[0,e]$ such a closed walk starts and ends outside $G_7$; hence no maximal visit
to $G_7$ crosses the boundary of the interval. We may therefore collapse each
maximal visit to $G_7$ into a kernel $K_{a,b}$ as in \eqref{eq:G7-kernel}. The
scalar weights omitted from the projective kernels merely multiply individual
path contributions and do not affect line containment. By
Lemmas~\ref{lem:two-state-frame} and~\ref{lem:G7-kernel-frame}, every elementary
block and every collapsed $G_7$ kernel preserves the two moving projective
lines. Hence the total pullback matrix $N_i$ of all such closed walks based in
$G_i$ satisfies
\[
N_ip_e^{(i)}\in\langle p_0^{(i)}\rangle,
\qquad
N_i\bar p_e^{(i)}\in\langle\bar p_0^{(i)}\rangle.
\]
Here the statement applies to the total matrix because the set of linear maps
sending each of two fixed lines into the corresponding target lines is closed
under addition. By \eqref{eq:eta-boundary},
\[
\eta_{e+s_i}=\eta_{s_i}+1,
\]
so
\[
p_e^{(i)}=\bar p_0^{(i)},
\qquad
\bar p_e^{(i)}=p_0^{(i)}.
\]
The two lines $\langle p_0^{(i)}\rangle$ and $\langle\bar p_0^{(i)}\rangle$ are
distinct, because $\eta_{s_i}^{-1}\ne(\eta_{s_i}+1)^{-1}$. Therefore $N_i$ is
anti-diagonal in the basis $(p_0^{(i)},\bar p_0^{(i)})$, and hence has trace
zero. Thus all closed walks whose chosen base point lies in
$G_1,\ldots,G_6$ have total contribution zero.

It remains to treat base points in $G_7$. The walks that never leave $G_7$ have
zero contribution by Lemma~\ref{lem:G7-trace-zero}, because their matrix is a
nonzero scalar multiple of a product of the matrices $R(C_j)$; the scalar does
not change the fact that the trace is zero.
Now consider the walks based in $G_7$ which do leave $G_7$. Partition them by the
first exit from $G_7$ after the chosen base point; every such exit goes from
$G_7$ to $G_3$. Suppose that the exit edge arrives at its $G_3$ endpoint at time
$k$ after the chosen base point. Then $1\le k<e$. Rotate each closed word so
that this $G_3$ endpoint is the new base state. The rotated word is therefore a
constrained $G_3$-based word on the interval from time $k$ to time $k+e$. This
rotation does not change the scalar weight of the word. We keep the resulting
constraints; we do not replace the constrained family by all $G_3$-based walks.
For a fixed first exit, this rotation is a bijection between the original
$G_7$-based closed words in that class and the corresponding constrained
$G_3$-based closed words on the rotated interval. The constraints only record
which excursion through $G_7$ crosses the right endpoint of the interval; they
do not change the fact that each resulting product is a composition of the
two-state blocks from Lemma~\ref{lem:two-state-frame} and of kernels covered by
Lemma~\ref{lem:G7-kernel-frame}.

Recall that a kernel $K_{a,b}$ represents an excursion that enters $G_7$ from
$G_2$ along the digit-$a$ edge and exits from $G_7$ to $G_3$ along the digit-$b$
edge; hence it maps the $G_3$ frame at time $b+1$ to the $G_2$ frame at time
$a$. After the rotation, all ordinary excursions through $G_7$ that are wholly
inside the interval collapse to such kernels. There is also one distinguished
excursion which crosses the end of the interval: the path enters $G_7$ from
$G_2$ at some time $a$ and exits at the endpoint time $k+e$, i.e. along the
edge of digit $k+e-1$. This segment is represented, up to its nonzero scalar weight, by the formal kernel
$K_{a,k+e-1}$. Lemma~\ref{lem:G7-kernel-frame} applies to this kernel as well,
because its identities hold for arbitrary integer indices.

When an index is increased by $e$, the coefficients $L_j,C_j,D_j$ are unchanged,
while $\eta_{j+e}=\eta_j+1$ exactly exchanges the barred and unbarred frames.
Thus the frame relations are consistent for every rotation amount $k$.
Consequently every product matrix belonging to this constrained rotated class
sends
\[
p_{k+e}^{(3)}\quad\text{into}\quad \langle p_k^{(3)}\rangle,
\qquad
\bar p_{k+e}^{(3)}\quad\text{into}\quad \langle\bar p_k^{(3)}\rangle.
\]
The same is true after summing over all choices inside the constrained class.
Since $p_{k+e}^{(3)}=\bar p_k^{(3)}$ and $\bar p_{k+e}^{(3)}=p_k^{(3)}$, the total
matrix of this class is anti-diagonal in the basis $(p_k^{(3)},\bar p_k^{(3)})$
and has trace zero. Summing over the possible first exits proves that the entire
contribution of $G_7$-based walks which leave $G_7$ is zero.

Combining the three cases proves that the large component $\mathcal C_1$ has
zero total contribution.
\end{proof}

\begin{proof}[Proof of Proposition~\ref{prop:intrinsic-hasse}]
By Lemma~\ref{lem:first-hasse-term}, the desired congruence is obtained by
computing the coefficient of $q$ in the minimal Hasse contribution.
Lemma~\ref{lem:minimal-carry-graph} gives the three directed cycle components.
Their contributions to that coefficient are
\[
\mathcal C_0:1,
\qquad
\mathcal C_2:\delta^{N/3}+\delta^{2N/3},
\qquad
\mathcal C_1:0
\]
by Lemma~\ref{lem:large-component-zero}. Therefore
\[
\mathcal F_{\delta,\Lambda}
\equiv
q\left(1+\delta^{N/3}+\delta^{2N/3}\right)
\pmod{2q},
\]
which is exactly \eqref{eq:intrinsic-hasse-congruence}.
\end{proof}

\section{Walsh values on $\fq$}

\begin{proof}[Proof of Theorem~\ref{thm:main-Fq}]
Let $\beta\in\fq$.
By Lemma~\ref{lem:basic-reduction},
\[
W_{f_{\alpha}}(\beta)=
\sum_{(y,z)\in\fq\times\mu_{q+1}}
\chi_q(\alpha(z^9+z^{-9})y^3+\beta(z^3+z^{-3})y)-q.
\]
For $\beta\in\fq$ we have $P_{\alpha}(1)=Q_{\beta}(1)=0$, so the slice $z=1$ contributes exactly $q$, which cancels the final $-q$.
Hence
\[
W_{f_{\alpha}}(\beta)
=
\sum_{(y,z)\in\fq\times\mu_{q+1}^*}
\chi_q(\alpha(z^9+z^{-9})y^3+\beta(z^3+z^{-3})y).
\]
Since $\gcd(3,q+1)=1$, the map $z\mapsto z^3$ permutes $\mu_{q+1}$.
Therefore
\begin{equation}\label{eq:thm12-start}
W_{f_{\alpha}}(\beta)
=
\sum_{(y,z)\in\fq\times\mu_{q+1}^*}
\chi_q(\alpha(z^3+z^{-3})y^3+\beta(z+z^{-1})y).
\end{equation}
Put
\[
T=\set{z+z^{-1}:z\in\mu_{q+1}^*}.
\]
By Lemma~\ref{lem:T-set},
\[
T=\set{t\in\fq^*: \Tr_q(1/t)=1}.
\]
Fix $\lambda\in\fq$ with $\Tr_q(\lambda)=1$.
For $t\in T$, the equation
\[
x^2+x+\lambda=1+t^{-2}
\]
has exactly two distinct roots in $\fq$: indeed
\[
\Tr_q(1+t^{-2}+\lambda)=\Tr_q(1)+\Tr_q(t^{-2})+\Tr_q(\lambda)=0+1+1=0,
\]
because $e$ is even and $\Tr_q(t^{-2})=\Tr_q(t^{-1})=1$. The map
$t\mapsto 1+t^{-2}$ is injective on $T$: if $1+t^{-2}=1+s^{-2}$, then
$t^{-1}=s^{-1}$ because the square map is an automorphism of $\fq$, and hence
$t=s$. Thus the resulting two-root fibres are disjoint; since they have total
size $q$, they form a partition of $\fq$.
Using again that the map $z\mapsto z+z^{-1}$ from $\mu_{q+1}^*$ onto $T$ is two-to-one, we get from \eqref{eq:thm12-start}
\begin{align*}
W_{f_{\alpha}}(\beta)
&=2\sum_{(y,t)\in\fq\times T}\chi_q(\alpha y^3t^3(1+t^{-2})+\beta ty)\\
&=2\sum_{(y,t)\in\fq\times T}\chi_q(\alpha y^3(1+t^{-2})+\beta y),
\end{align*}
where in the second line we replaced $y$ by $t^{-1}y$ for each fixed $t\in T$.
Since each value $1+t^{-2}$ with $t\in T$ is attained twice by the map $x\mapsto x^2+x+\lambda$, we obtain
\begin{align*}
W_{f_{\alpha}}(\beta)
&=\sum_{x\in\fq}\sum_{y\in\fq}\chi_q(\alpha y^3(x^2+x+\lambda)+\beta y)\\
&=\sum_{y\in\fq}\chi_q(\alpha\lambda y^3+\beta y)
\sum_{x\in\fq}\chi_q(\alpha x^2y^3+\alpha xy^3).
\end{align*}
Since
\[
\Tr_q(\alpha xy^3)=\Tr_q\bigl((\alpha xy^3)^2\bigr)=\Tr_q(\alpha^2x^2y^6),
\]
we obtain
\begin{equation}\label{eq:thm12-key}
W_{f_{\alpha}}(\beta)
=
\sum_{y\in\fq}\chi_q(\alpha\lambda y^3+\beta y)
\sum_{x\in\fq}\chi_q(\alpha x^2y^3(1+\alpha y^3)).
\end{equation}

Assume first that $\alpha$ is not a cube in $\fq$.
Then $1+\alpha y^3\neq 0$ for every $y\in\fq^*$.
For such $y$, the map $x\mapsto \alpha x^2y^3(1+\alpha y^3)$ is a permutation of $\fq$, and therefore
\[
\sum_{x\in\fq}\chi_q(\alpha x^2y^3(1+\alpha y^3))=\sum_{x\in\fq}\chi_q(x)=0.
\]
The only surviving term in \eqref{eq:thm12-key} is $y=0$, which contributes $q$.
Hence
\[
W_{f_{\alpha}}(\beta)=q
\qquad\text{for all }\beta\in\fq.
\]
This proves part~\textup{(b)}.

Now assume that $\alpha$ is a cube in $\fq$.
Since $\gcd(3,q-1)=3$, the equation $\alpha y^3=1$ has exactly three roots
$y_1,y_2,y_3\in\fq^*$, and these roots satisfy
\[
y_2=\omega y_1,
\qquad
y_3=\omega^2 y_1,
\qquad
y_1+y_2+y_3=0,
\]
where $\omega\in\fq$ is a primitive cubic root of unity.
Equation \eqref{eq:thm12-key} now gives
\begin{align*}
W_{f_{\alpha}}(\beta)
&=q\sum_{y\in\{0,y_1,y_2,y_3\}}\chi_q(\alpha\lambda y^3+\beta y)\\
&=q\Bigl(1+(-1)^{1+\Tr_q(\beta y_1)}+(-1)^{1+\Tr_q(\beta y_2)}+(-1)^{1+\Tr_q(\beta y_3)}\Bigr),
\end{align*}
because $\Tr_q(\lambda)=1$ and all other $y$ contribute zero.

For any $\beta\in\fq$, the triple
\[
\bigl(\Tr_q(\beta y_1),\Tr_q(\beta y_2),\Tr_q(\beta y_3)\bigr)\in\ftwo^3
\]
has even parity, since
\[
\Tr_q(\beta y_1)+\Tr_q(\beta y_2)+\Tr_q(\beta y_3)
=\Tr_q\bigl(\beta(y_1+y_2+y_3)\bigr)=0.
\]
Hence the only possibilities are $(0,0,0)$ or a triple with exactly two coordinates equal to $1$.
Define
\begin{align*}
\mathcal W_1&=\set{\beta\in\fq:\Tr_q(\beta y_1)=\Tr_q(\beta y_2)=\Tr_q(\beta y_3)=0},\\
\mathcal W_2&=\set{\beta\in\fq:\bigl(\Tr_q(\beta y_1),\Tr_q(\beta y_2),\Tr_q(\beta y_3)\bigr)\text{ has exactly two entries }1}.
\end{align*}
Then $\fq=\mathcal W_1\cup\mathcal W_2$ and
\[
W_{f_{\alpha}}(\beta)=
\begin{cases}
-2q,& \beta\in\mathcal W_1,\\
\phantom{-}2q,& \beta\in\mathcal W_2.
\end{cases}
\]

Since $y_2=\omega y_1$ with $\omega\notin\ftwo$, the elements $y_1$ and $y_2$ are $\ftwo$-linearly independent.
Therefore
\[
\mathcal W_1=\set{\beta\in\fq:\Tr_q(\beta y_1)=\Tr_q(\beta y_2)=0}
\]
is an $(e-2)$-dimensional $\ftwo$-subspace of $\fq$.
Hence
\[
\abs{\mathcal W_1}=q/4,
\qquad
\abs{\mathcal W_2}=3q/4.
\]
In particular, both sets are nonempty, and therefore
\[
\set{W_{f_{\alpha}}(\beta):\ \beta\in\fq}=\set{2q,-2q}.
\]
This proves part~\textup{(a)} and completes the proof of Theorem~\ref{thm:main-Fq}.
\end{proof}

\section{Walsh values off $\fq$}

\begin{proof}[Proof of Theorem~\ref{thm:main-outside}]
Let $\beta\in\fqq\setminus\fq$ and put $b=\beta+\bar\beta\in\fq^*$.
Choose $\lambda\in\fq$ with $\Tr_q(\lambda)=1$, choose $\theta\in\fqq$ with
$\theta^2+\theta=\lambda+1$, and write
\[
\beta=b(c+\theta),
\qquad c\in\fq.
\]
Put
\[
\kappa=\alpha b^{-3},
\qquad
\Lambda=c^2+c+\lambda.
\]
Since $b^{-3}$ is a cube in $\fq^*$, the element $\kappa$ is a noncube whenever $\alpha$ is a noncube.
By Proposition~\ref{prop:outside-reduction},
\[
W_{f_\alpha}(\beta)+q=\mathcal F_{\kappa,\Lambda}.
\]
Proposition~\ref{prop:intrinsic-hasse} gives
\[
\mathcal F_{\kappa,\Lambda}
\equiv
q\left(1+\kappa^{N/3}+\kappa^{2N/3}\right)
\pmod {2q},
\qquad N=q-1.
\]
Because $\kappa$ is a noncube, $\omega=\kappa^{N/3}$ belongs to $\mathbb F_4^*\setminus\{1\}$, and hence
\[
1+\omega+\omega^2=0.
\]
Thus
\[
\mathcal F_{\kappa,\Lambda}\equiv0\pmod {2q},
\]
and therefore
\[
W_{f_\alpha}(\beta)\equiv q\pmod {2q}.
\]
Walsh coefficients are ordinary integers, and the last congruence says precisely
that every outside Walsh coefficient is of the form $q(2m+1)$ with
$m\in\mathbb Z$, that is, an odd multiple of $q$.

By Theorem~\ref{thm:main-Fq}\textup{(b)}, $W_{f_\alpha}(\gamma)=q$ for every $\gamma\in\fq$. Walsh orthogonality gives
\[
\sum_{\gamma\in\fqq}W_{f_\alpha}(\gamma)^2=q^4.
\]
Hence
\[
\sum_{\gamma\in\fqq\setminus\fq}W_{f_\alpha}(\gamma)^2
=q^4-q\cdot q^2=q^3(q-1).
\]
There are $q^2-q=q(q-1)$ elements in $\fqq\setminus\fq$. Each outside coefficient is an odd multiple of $q$, so each outside square is at least $q^2$. The total square sum is exactly $q(q-1)q^2$, and therefore every outside square is equal to $q^2$. Thus
\[
W_{f_\alpha}(\beta)\in\{q,-q\}
\qquad(\beta\in\fqq\setminus\fq).
\]
Finally, since $f_\alpha(0)=0$, the first Walsh orthogonality relation gives
\[
\begin{aligned}
\sum_{\gamma\in\fqq}W_{f_\alpha}(\gamma)
&=\sum_{x\in\fqq}(-1)^{f_\alpha(x)}
  \sum_{\gamma\in\fqq}(-1)^{\Tr_{q^2}(\gamma x)} \\
&=q^2(-1)^{f_\alpha(0)}=q^2.
\end{aligned}
\]
The contribution from $\fq$ is already $q^2$, so the outside coefficients have total sum zero. Since there are outside points and each outside coefficient is $\pm q$, both signs occur. This proves Theorem~\ref{thm:main-outside}.
\end{proof}

\begin{proof}[Proof of Corollary~\ref{cor:distribution}]
Assume that $\alpha$ is not a cube in $\fq$.
By Theorem~\ref{thm:main-Fq}\textup{(b)}, the value $q$ occurs at every point of $\fq$.
By Theorem~\ref{thm:main-outside}, every point of $\fqq\setminus\fq$ contributes either $q$ or $-q$.
Let $N_+$ and $N_-$ be the total multiplicities of $q$ and $-q$, respectively. Then
\[
N_++N_-=q^2,
\qquad
q(N_+-N_-)=\sum_{\beta\in\fqq}W_{f_\alpha}(\beta)=q^2,
\]
where the last equality is the Walsh orthogonality identity just used above.
Thus
\[
N_+=\frac{q(q+1)}2,
\qquad
N_-=\frac{q(q-1)}2.
\]
\end{proof}

\begin{proof}[Proof of Corollary~\ref{cor:conj}]
If $\alpha$ is not a cube in $\fq$, then Theorem~\ref{thm:main-Fq}\textup{(b)} and Theorem~\ref{thm:main-outside} show that every Walsh coefficient has absolute value $q$, so $f_\alpha$ is bent.
If $\alpha$ is a cube in $\fq$, then Theorem~\ref{thm:main-Fq}\textup{(a)} gives Walsh coefficients $\pm2q$ on $\fq$, and hence $f_\alpha$ is not bent.
Therefore Conjecture~\ref{con:main} holds.
\end{proof}

\begin{rem}\label{rem:cubic-plateaued}
Although the cubic case is already sufficient to disprove bentness by
Theorem~\ref{thm:main-Fq}\textup{(a)}, computations suggest a more precise
spectral behaviour. Namely, if $\alpha$ is a cube in $\fq$, then
$$
\set{W_{f_{\alpha}}(\beta):\ \beta\in\fqq\setminus\fq}=\begin{cases} \set{0},& e=2,\\
\set{-2q,0,2q},& e\ge 4.
\end{cases}$$
The methods of the present paper prove the weaker divisibility
$W_{f_\alpha}(\beta)\equiv0\pmod{2q}$ for $\beta\notin\fq$ in the cubic case,
but do not by themselves exclude larger multiples of $2q$.
\end{rem}

\section{Cyclotomic reformulation and application}

In this section we show that the cyclotomic family introduced by Xie, Li, Wang, and Zeng~\cite[Conjecture~1]{Xie2025} coincides with the present permutation-inverse family.
Thus their conjecture is not a separate phenomenon: the two constructions define identical Boolean functions after the coordinate change below, not merely EA- or CCZ-equivalent functions. It is therefore a direct reformulation of Corollary~\ref{cor:conj}.
Let $u$ be a generator of $\mu_{q+1}$. Since $\fqq^*=\fq^*\mu_{q+1}$ and
$\fq^*\cap\mu_{q+1}=\{1\}$, the cosets $u^i\fq^*$, $0\le i\le q$, form a
disjoint decomposition of $\fqq^*$.
For $\alpha\in\fq^*$, define a Boolean function $g_\alpha$ on $\fqq$ by
\[
 g_\alpha(0)=0,
 \qquad
 g_\alpha(x)=\Tr_{q^2}\!\left(
 \alpha\,\frac{u^{6i}}{(1+u^{2i}+u^{-2i})^3}x^3
 \right)
 \quad\text{for }x\in u^i\fq^*,\ 0\le i\le q.
\]
The conjecture in~\cite[Conjecture~1]{Xie2025} asserts the following.
\begin{con}\label{con:xie}
$g_\alpha$ is bent if and only if
$\alpha$ is not a cube in $\fq$.
\end{con}
We prove that $g_\alpha$ coincides with the function $f_\alpha$ defined in~\eqref{eq:def-falpha} for every $\alpha\in\fq^*$; this proves Conjecture~\ref{con:xie}.

\begin{thm}\label{thm:cyclotomic-application}
For every $z\in \mu_{q+1}$ and every $x\in z\fq^*$, we have
\[
 \sigma^{-1}(x)=\frac{z^2}{1+z^2+z^{-2}}\,x.
\]
In particular,
\[
 g_\alpha(x)=f_\alpha(x)\qquad\text{for all }x\in \fqq.
\]
\end{thm}

\begin{proof}
Fix $z\in \mu_{q+1}$ and put
\[
 \ell_z=1+z^2+z^{-2}.
\]
Since $z^q=z^{-1}$, we have
\[
 \ell_z^q=1+z^{-2}+z^2=\ell_z,
\]
so $\ell_z\in \fq$. We claim that $\ell_z\ne 0$. Indeed, if
$1+z^2+z^{-2}=0$, then $z^4+z^2+1=0$, and hence
\[
 z^6+1=(z^2+1)(z^4+z^2+1)=0.
\]
Thus $z^6=1$. Since $e$ is even, we have $q=2^e\equiv 1\pmod 3$, so
$q+1\equiv 2\pmod 3$ and therefore $\gcd(3,q+1)=1$. Also $q+1$ is odd, so
$\gcd(2,q+1)=1$. Hence $\gcd(6,q+1)=1$. Because $z\in \mu_{q+1}$ and
$z^6=1$, it follows that $z=1$, which is impossible because
$1+1+1=1\ne 0$ in characteristic $2$. Therefore $\ell_z\in \fq^*$.

Now let $x\in z\fq^*$ and write $x=cz$ with $c\in \fq^*$. Set
\[
 y=\frac{c}{\ell_z}z^3.
\]
Since $e$ is even, $q\equiv 1\pmod 3$, and hence
\[
 d=\frac{q^2+q+1}{3}=1+\frac{q+2}{3}(q-1).
\]
Therefore $a^d=a$ for every $a\in \fq$. Also,
\[
 z^{3d}=z^{q^2+q+1}=z
\]
because $z^{q+1}=1$. It follows that
\[
 y^d=\left(\frac{c}{\ell_z}\right)^d z^{3d}=\frac{c}{\ell_z}z,
 \qquad
 y^{dq}=(y^d)^q=\frac{c}{\ell_z}z^{-1}.
\]
Consequently,
\begin{align*}
 \sigma(y)
 &=y+y^d+y^{dq}\\
 &=\frac{c}{\ell_z}(z^3+z+z^{-1})\\
 &=\frac{c}{\ell_z}\,z(1+z^2+z^{-2})\\
 &=cz=x.
\end{align*}
Since $\sigma$ is a permutation of $\fqq$, we obtain
\[
 \sigma^{-1}(x)=y=\frac{c}{\ell_z}z^3=\frac{z^2}{1+z^2+z^{-2}}\,x.
\]
This proves the first assertion.

Now let $x\in \fqq^*$. There is a unique integer $i$ with $0\le i\le q$
such that $x\in u^i\fq^*$. Applying the formula just proved with $z=u^i$, we get
\[
 \sigma^{-1}(x)=\frac{u^{2i}}{1+u^{2i}+u^{-2i}}\,x.
\]
Cubing both sides yields
\[
 (\sigma^{-1}(x))^3=
 \frac{u^{6i}}{(1+u^{2i}+u^{-2i})^3}x^3.
\]
Hence
\[
 f_\alpha(x)=\Tr_{q^2}\bigl(\alpha(\sigma^{-1}(x))^3\bigr)
 =\Tr_{q^2}\!\left(
 \alpha\,\frac{u^{6i}}{(1+u^{2i}+u^{-2i})^3}x^3
 \right)
 =g_\alpha(x).
\]
Also $f_\alpha(0)=g_\alpha(0)=0$, since $\sigma^{-1}(0)=0$. Therefore
$g_\alpha(x)=f_\alpha(x)$ for all $x\in \fqq$. The proof of Theorem~\ref{thm:cyclotomic-application} is complete.
\end{proof}

\begin{cor}\label{cor:xie-conj}
Let $\alpha\in \fq^*$. Then the function $g_\alpha$ is bent if and only if
$\alpha$ is not a cube in $\fq$.
\end{cor}

\begin{proof}
By Theorem~\ref{thm:cyclotomic-application}, the functions $g_\alpha$ and
$f_\alpha$ are identical. Therefore, the assertion follows immediately from
Corollary~\ref{cor:conj}. This proves Conjecture~\ref{con:xie}, equivalently
Conjecture~1 of~\cite{Xie2025}.
\end{proof}

\appendix

\section{Finite checks for the Hasse congruence}\label{app:finite-checks}
This appendix records the finite verifications used in the proof of Proposition~\ref{prop:intrinsic-hasse}. These checks are finite certificates for the explicit combinatorial and matrix statements invoked in Section~2, not numerical tests for special finite fields. None of the checks depends on $q$ or on $e$; the only use of $e$ in the main proof is the cyclic boundary condition and the fact that $e$ is even. The ordinary Python scripts use only the standard library and were tested with Python 3.13.5; the final symbolic check was tested with SageMath 10.x. The expected outputs are displayed after the corresponding scripts.

The checks appear in the same order as they are used in the proof: potential inequality and equality components; weighted large-component blocks; trace-zero algebra; and the $G_7$ projective-frame identities. The first script verifies the potential inequality \eqref{eq:potential-ineq} and lists the directed cycle components of the equality graph.

\begin{verbatim}
from collections import defaultdict
D = [(3,1),(2,1),(1,1),(8,3),(1,3),(0,3)]
nu = [
[3,3,2,1,1,1,1,0,0,0,0,0,0,0,0],
[3,2,2,2,3,1,1,0,0,0,0,0,0,0,0],
[2,2,2,2,2,2,2,2,0,0,0,0,0,0,0],
[2,1,2,1,2,1,2,1,1,0,0,0,0,0,0],
[1,1,1,1,1,1,1,1,0,0,0,0,0,0,0],
[1,1,1,1,1,1,1,0,0,0,0,0,0,0,0],
[0]*15,[0]*15,[0]*15,[0]*15,[0]*15,[0]*15]
def pot(v):
    x,y = v
    return nu[y-1][x-1]
def dsum(mask):
    return (sum(D[i][0] for i in range(6) if mask>>i & 1),
            sum(D[i][1] for i in range(6) if mask>>i & 1))
def size(mask):
    return bin(mask).count("1")

bad, eq = [], []
for x in range(1,16):
  for y in range(1,13):
    phi = (x,y)
    for mask in range(64):
      a,b = dsum(mask)
      if (x+a)%2 or (y+b)%2:
          continue
      phip = ((x+a)//2,(y+b)//2)
      if not (1 <= phip[0] <= 15 and 1 <= phip[1] <= 12):
          continue
      slack = size(mask)-1-pot(phi)+pot(phip)
      if slack < 0:
          bad.append((phi,mask,phip,slack))
      if slack == 0:
          eq.append((phi,mask,phip))
assert bad == []

adj = defaultdict(list)
V = set()
for u,mask,v in eq:
    adj[u].append((v,mask)); V.add(u); V.add(v)

def reach_from(a):
    seen, stack = set(), [v for v,m in adj[a]]
    while stack:
        v = stack.pop()
        if v in seen:
            continue
        seen.add(v)
        stack += [w for w,m in adj[v]]
    return seen
R = {v: reach_from(v) for v in V}
cyclic = {v for v in V if v in R[v]}
components = []
while cyclic:
    v = next(iter(cyclic))
    comp = {w for w in cyclic if w in R[v] and v in R[w]}
    components.append(sorted(comp))
    cyclic -= comp
print(sorted(components, key=lambda c:(len(c),c)))
\end{verbatim}

The printed components are
\[
\begin{aligned}
&\relax[(3,1)],\qquad [(3,2),(6,4)],\\
&[(1,1),(1,2),(1,3),(2,1),(2,2),(2,3),(3,3),\\
&\quad (4,2),(4,3),(4,4),(5,2),(5,3),(5,4),(6,3),\\
&\quad (7,3),(7,4),(7,5),(8,3),(8,4),(8,5)].
\end{aligned}
\]
The same equality edges, with their coefficient weights, give the transition
matrices displayed before Lemma~\ref{lem:two-state-frame} and
Lemma~\ref{lem:G7-kernel-frame}. The following ordinary Python check rebuilds
those weighted blocks from the equality graph. A polynomial is represented by
the set of its monomials in the symbols $L,D,C$ over $\ftwo$; this makes the
verification independent of $q$, $e$, $\Lambda$, and $\delta$.

\begingroup\footnotesize
\begin{verbatim}
from collections import defaultdict
D = [(3,1),(2,1),(1,1),(8,3),(1,3),(0,3)]
nu = [
[3,3,2,1,1,1,1,0,0,0,0,0,0,0,0],
[3,2,2,2,3,1,1,0,0,0,0,0,0,0,0],
[2,2,2,2,2,2,2,2,0,0,0,0,0,0,0],
[2,1,2,1,2,1,2,1,1,0,0,0,0,0,0],
[1,1,1,1,1,1,1,1,0,0,0,0,0,0,0],
[1,1,1,1,1,1,1,0,0,0,0,0,0,0,0],
[0]*15,[0]*15,[0]*15,[0]*15,[0]*15,[0]*15]
def pot(v): return nu[v[1]-1][v[0]-1]
def dsum(mask):
    return (sum(D[i][0] for i in range(6) if mask>>i & 1),
            sum(D[i][1] for i in range(6) if mask>>i & 1))
def size(mask): return bin(mask).count('1')
eq = []
for x in range(1,16):
  for y in range(1,13):
    for mask in range(64):
      a,b = dsum(mask)
      if (x+a)%2 or (y+b)%2: continue
      xp,yp = (x+a)//2, (y+b)//2
      if not (1 <= xp <= 15 and 1 <= yp <= 12): continue
      if size(mask)-1-pot((x,y))+pot((xp,yp)) == 0:
          eq.append(((x,y),mask,(xp,yp)))
G = {
1:[(1,1),(2,1)], 2:[(1,2),(5,2)], 3:[(2,2),(4,2)],
4:[(4,4),(8,4)], 5:[(5,4),(7,4)], 6:[(7,5),(8,5)],
7:[(1,3),(2,3),(3,3),(4,3),(5,3),(6,3),(7,3),(8,3)]}
group = {v:g for g,vs in G.items() for v in vs}
C1 = set(group)
# A polynomial is a set of monomials L^a D^b C^c over F_2.
Z = frozenset()
def mono(a=0,b=0,c=0): return frozenset({(a,b,c)})
def add(p,q): return p ^ q
def mul(p,q):
    out = Z
    for a,b,c in p:
      for d,e,f in q:
        out = add(out, mono(a+d,b+e,c+f))
    return out
ONE,L,Dv,C = mono(), mono(1,0,0), mono(0,1,0), mono(0,0,1)
def weight(mask):
    w = ONE
    for i,f in enumerate([ONE,ONE,L,Dv,Dv,mul(Dv,C)]):
        if mask>>i & 1: w = mul(w,f)
    return w
def mat(rows, cols, entries=()):
    M = [[Z for _ in cols] for _ in rows]
    for r,c,v in entries: M[r][c] = v
    return M
def block(g,h):
    rows, cols = G[g], G[h]
    M = [[Z for _ in cols] for _ in rows]
    for u,mask,v in eq:
      if u in rows and v in cols:
        r, c = rows.index(u), cols.index(v)
        M[r][c] = add(M[r][c], weight(mask))
    return M
DC, D2, LC, DLC = mul(Dv,C), mul(Dv,Dv), mul(L,C), mul(mul(Dv,L),C)
expected = {
(1,1): mat(G[1],G[1], [(0,0,L),(0,1,ONE),(1,1,ONE)]),
(1,2): mat(G[1],G[2], [(0,0,Dv),(1,0,DC),(1,1,Dv)]),
(2,3): mat(G[2],G[3], [(0,0,L),(1,1,L)]),
(2,5): mat(G[2],G[5], [(0,0,D2),(1,1,D2)]),
(3,1): mat(G[3],G[1], [(0,0,ONE),(1,1,ONE)]),
(4,3): mat(G[4],G[3], [(0,0,ONE),(1,1,ONE)]),
(5,4): mat(G[5],G[4], [(0,0,add(Dv,DC)),(0,1,Dv),(1,0,DLC),(1,1,mul(Dv,L))]),
(5,6): mat(G[5],G[6], [(0,0,D2),(1,1,D2)]),
(6,4): mat(G[6],G[4], [(0,0,Dv),(1,0,DC),(1,1,Dv)]),
(2,7): mat(G[2],G[7], [(0,0,DLC),(0,1,add(Dv,DC)),(0,4,mul(Dv,L)),(0,5,Dv),
                            (1,2,DLC),(1,3,add(Dv,DC)),(1,6,mul(Dv,L)),(1,7,Dv)]),
(7,3): mat(G[7],G[3], [(0,0,ONE),(1,0,ONE),(2,0,L),(4,1,ONE),(5,1,ONE),(6,1,L)]),
}
# internal G7 block: D*(P+C Q)
P = [(0,0),(1,4),(2,1),(3,5),(4,2),(5,6),(6,3),(7,7)]
Q = [(1,0),(3,1),(5,2),(7,3)]
expected[(7,7)] = mat(G[7],G[7], [(r,c,Dv) for r,c in P] + [(r,c,DC) for r,c in Q])
actual_pairs = {(group[u], group[v]) for u,m,v in eq if u in C1 and v in C1}
assert actual_pairs == set(expected)
for pair, M in expected.items():
    assert block(*pair) == M, pair
print('all large-component blocks verified')
\end{verbatim}
\endgroup
The expected output is
\begin{verbatim}
all large-component blocks verified
\end{verbatim}

The next ordinary Python check verifies the trace-zero algebra used in Lemma~\ref{lem:G7-trace-zero}. Matrices are over $\ftwo$, so the trace of $I_8$ is $8=0$.

\begin{verbatim}
def add(A,B):
    return [[A[i][j]^B[i][j] for j in range(8)] for i in range(8)]
def mul(A,B):
    return [[sum(A[i][k]&B[k][j] for k in range(8))%2
             for j in range(8)] for i in range(8)]
def key(A):
    return tuple(tuple(r) for r in A)
def tr(A):
    return sum(A[i][i] for i in range(8))%2
I = [[1 if i==j else 0 for j in range(8)] for i in range(8)]
P = [[1,0,0,0,0,0,0,0],
     [0,0,0,0,1,0,0,0],
     [0,1,0,0,0,0,0,0],
     [0,0,0,0,0,1,0,0],
     [0,0,1,0,0,0,0,0],
     [0,0,0,0,0,0,1,0],
     [0,0,0,1,0,0,0,0],
     [0,0,0,0,0,0,0,1]]
Q = [[0,0,0,0,0,0,0,0],
     [1,0,0,0,0,0,0,0],
     [0,0,0,0,0,0,0,0],
     [0,1,0,0,0,0,0,0],
     [0,0,0,0,0,0,0,0],
     [0,0,1,0,0,0,0,0],
     [0,0,0,0,0,0,0,0],
     [0,0,0,1,0,0,0,0]]
basis = [I,P,Q,mul(P,P),mul(P,Q),mul(Q,P),mul(Q,Q),
         mul(mul(P,P),Q),mul(mul(P,Q),P),mul(mul(P,Q),Q),
         mul(mul(Q,P),P),mul(mul(Q,P),Q),mul(mul(Q,Q),P),
         mul(mul(Q,Q),Q),mul(mul(mul(P,P),Q),P),
         mul(mul(P,P),mul(Q,Q))]
span = {key([[0]*8 for _ in range(8)])}
for B in basis:
    span |= {key(add([list(r) for r in A],B)) for A in list(span)}
assert len(span) == 2**16
assert all(tr(B) == 0 for B in basis)
assert all(key(mul(B,P)) in span for B in basis)
assert all(key(mul(B,Q)) in span for B in basis)
print('trace-zero algebra verified')
\end{verbatim}

The expected output is
\begin{verbatim}
trace-zero algebra verified
\end{verbatim}

Finally, the projective identities in Lemma~\ref{lem:G7-kernel-frame} can be
checked symbolically over the rational function field $\ftwo(t)$. The following
SageMath code, tested with SageMath 10.x, verifies the three-dimensional
invariance and the three boundary annihilations. The barred identities are
checked by replacing $t$ with $t+1$; this is exactly the conjugation
$\eta\mapsto\eta+1$ used in the proof. Since the calculation takes place in a
fraction field, all identities may equivalently be cleared of denominators;
only powers of $t$ and $t+1$ occur as possible denominators.

\begin{verbatim}
R.<t> = PolynomialRing(GF(2)); K = FractionField(R)
def e(j): return K(t)^(2^j)
def L(j): return e(j+1)+e(j)
def C(j): return e(j+3)+e(j)
def col(v): return vector(K,v)
P = Matrix(K,[[1,0,0,0,0,0,0,0],
[0,0,0,0,1,0,0,0],[0,1,0,0,0,0,0,0],
[0,0,0,0,0,1,0,0],[0,0,1,0,0,0,0,0],
[0,0,0,0,0,0,1,0],[0,0,0,1,0,0,0,0],
[0,0,0,0,0,0,0,1]])
Q = Matrix(K,[[0,0,0,0,0,0,0,0],
[1,0,0,0,0,0,0,0],[0,0,0,0,0,0,0,0],
[0,1,0,0,0,0,0,0],[0,0,0,0,0,0,0,0],
[0,0,1,0,0,0,0,0],[0,0,0,0,0,0,0,0],
[0,0,0,1,0,0,0,0]])
def Rj(j): return P + C(j)*Q
def T(j):
    return Matrix(K,[[1,0],[1,0],[L(j),0],[0,0],
                     [0,1],[0,1],[0,L(j)],[0,0]])
def S(j):
    return Matrix(K,[[L(j)*C(j),1+C(j),0,0,L(j),1,0,0],
                     [0,0,L(j)*C(j),1+C(j),0,0,L(j),1]])
def u(j,bar=False):
    x = 1/(e(j+2)+(1 if bar else 0))
    return T(j)*col([x,1])
def in_span(v,cols):
    return Matrix(K,cols+[v]).rank() == Matrix(K,cols).rank()
for bar in [False,True]:
    B0 = [u(0,bar), Rj(0)*u(1,bar), Rj(0)*Rj(1)*u(2,bar)]
    assert in_span(Rj(0)*u(1,bar), B0)
    assert in_span(Rj(0)*Rj(1)*u(2,bar), B0)
    assert in_span(Rj(0)*Rj(1)*Rj(2)*u(3,bar), B0)
    x = 1/(e(2)+(1 if bar else 0))
    alpha = Matrix(K,[[1,x]])*S(0)
    assert alpha*u(1,bar) == 0
    assert alpha*Rj(1)*u(2,bar) == 0
    assert alpha*Rj(1)*Rj(2)*u(3,bar) == 0
print('G7-kernel frame identities verified')
\end{verbatim}
The expected output is
\begin{verbatim}
G7-kernel frame identities verified
\end{verbatim}

\section*{Acknowledgments}
The author thanks Domingo Gom\'{e}z P\'{e}rez (University of Cantabria) for helpful discussions and computational checks related to the finite-field reductions.

\end{document}